\documentclass[a4paper,12pt,reqno]{amsart}
\usepackage{amsfonts}
\usepackage{amsmath}
\usepackage{amssymb}
\usepackage[a4paper]{geometry}
\usepackage{mathrsfs}
\usepackage{csquotes}
\usepackage{enumitem}

\usepackage{xcolor}
\usepackage[colorlinks = true,
            linkcolor = red,
            urlcolor  = gray,
            citecolor = blue,
            anchorcolor = blue]{hyperref}
\renewcommand\eqref[1]{(\ref{#1})} 
\numberwithin{equation}{section}
\theoremstyle{plain}
\newtheorem{theorem}{Theorem}[section]
\newtheorem{proposition}[theorem]{Proposition}
\newtheorem{corollary}[theorem]{Corollary}
\newtheorem{lemma}[theorem]{Lemma}
\theoremstyle{definition}
\newtheorem{definition}[theorem]{Definition}
\newtheorem{remark}[theorem]{Remark}

\DeclareMathOperator*{\esssup}{ess\,sup}
\DeclareMathOperator*{\essinf}{ess\,inf}

\begin{document}

   \title[Subelliptic $p$-Laplacian spectral problem  for H\"ormander vector fields]
   {Subelliptic $p$-Laplacian spectral problem \\ for H\"ormander vector fields}

\author[M. Karazym]{Mukhtar Karazym}
\address{
  Mukhtar Karazym:
  \endgraf
  Department of Mathematics
  \endgraf
Nazarbayev University, Kazakhstan
  \endgraf
  {\it E-mail address} {\rm  mukhtar.karazym@nu.edu.kz; mukhtarkarazym@gmail.com}
  }

 \author[D. Suragan]{Durvudkhan Suragan}
\address{
	Durvudkhan Suragan:
	\endgraf
	Department of Mathematics
	\endgraf
Nazarbayev University, Kazakhstan
	\endgraf
	{\it E-mail address} {\rm durvudkhan.suragan@nu.edu.kz}
}

\keywords{Spectral problem, H\"ormander vector fields,  subelliptic $p$-Laplacian}
\subjclass[2020]{35P30, 35H20, 35J92, 35R03}

\begin{abstract} 
Based on variational methods, we study the spectral problem for the subelliptic $p$-Laplacian  arising from smooth H\"ormander vector fields. We derive the smallest eigenvalue, prove its simplicity and  isolatedness, establish  the positivity of the first eigenfunction and show  H\"older regularity of eigenfunctions with respect to the control distance.  Moreover, we determine the best constant for the $L^{p}$-Poincar\'e-Friedrichs inequality for H\"ormander vector fields as a byproduct.
\end{abstract}
\maketitle


\section{Introduction}
The following nonlinear eigenvalue problem represents the simplest prototype of our objective:
\begin{equation}\label{Dirichlet p-Laplacian}
\begin{aligned}
\operatorname{div}\left(|\nabla u|^{p-2} \nabla u\right)&=-\lambda|u|^{p-2} u  &&\text { in } \Omega, \\
u&=0 &&\text { on } \partial \Omega,
\end{aligned}
\end{equation}
where $1<p<\infty$, $\nabla$ is the usual gradient and $\Omega\subset\mathbb{R}^{n}$ is a bounded domain.  Extensive literature exists on the problem \eqref{Dirichlet p-Laplacian}, see, for example, \cite{Bonder2015review}, \cite{lindqvist2008nonlinear} and references therein. 

However, relatively limited research has been conducted on  eigenvalue problems for the $p$-sub-Laplacian on  stratified Lie groups $\mathbb{G}=\left(\mathbb{R}^{n}, \circ\right)$ (with the group operation $\circ$) 
\begin{equation}\label{Dirichlet p-sub-Laplacian}
\begin{aligned}
\operatorname{div}_{\mathbb{G}}\left(\left|\nabla_{\mathbb{G}} u\right|^{p-2} \nabla_{\mathbb{G}} u\right)&=-\lambda|u|^{p-2} u  &&\text { in } \Omega\subset \mathbb{G}, \\
u&=0 &&\text { on } \partial \Omega,
\end{aligned}
\end{equation}
where $\nabla_{\mathbb{G}}$ and  $\operatorname{div}_{\mathbb{G}} v:=\nabla_{\mathbb{G}} \cdot v$ are the horizontal gradient and  horizontal divergence on $\mathbb{G}$, respectively. In the previous study \cite{WNL2009}, several estimates of the first and second eigenvalues for the Dirichlet $p$-sub-Laplacian \eqref{Dirichlet p-sub-Laplacian} were obtained. Recently, Carfagnini and Gordina \cite{ CG2024} focused on the linear case of the spectral problem \eqref{Dirichlet p-sub-Laplacian} when $p=2$, thereby advancing our understanding of this topic within stratified Lie groups.  Chen and Chen \cite{Chen2021}, assuming only the H\"ormander finite rank condition for smooth vector fields on $\Omega$, studied the eigenvalue problem in Carnot-Carath\'eodory spaces
\begin{equation*}
\begin{aligned}
\Delta_{X} u&=-\lambda u  &&\text { in } \Omega, \\
u&=0 &&\text { on } \partial \Omega.
\end{aligned}
\end{equation*}
It is worth mentioning that, when $p=2$, the best constant of the Poincar\'e-Friedrichs inequality for H\"ormander vector fields was found in \cite{Chen2021}. In this paper, we extend the finding in \cite{Chen2021} and establish the best constant for the entire range $1<p<\infty$ (see Corollary \ref{the best constant of Poincare}). 

Thus, the present paper aims to study the following eigenvalue problem for the subelliptic $p$-Laplacian in Carnot-Carath\'eodory spaces
\begin{equation}\label{Dirichlet p-sub-Laplacian type}
\begin{aligned}
\sum_{i=1}^{m}X_{i}^{*}\left(\left|X u\right|^{p-2} X_{i}u\right)&=\lambda|u|^{p-2} u  &&\text { in } \Omega, \\
u&=0 &&\text { on } \partial \Omega,
\end{aligned}
\end{equation}
where $X_{1}, \ldots, X_{m}$  are smooth  H\"ormander vector fields, $X_{i}^{*}$ is the formal adjoint of $X_{i}$ and $m\leq n$. This type of problem was considered in \cite{RSS2021}, where the authors, for example, studied the simplicity of the first eigenvalue of \eqref{Dirichlet p-sub-Laplacian type} by means of the Picone identity, whereas we prove it by choosing suitable test functions.

The findings presented in this paper contribute to the understanding and analysis of eigenvalue problems for the quasilinear operators \eqref{Dirichlet p-sub-Laplacian type} and provide insights into the behavior of eigenfunctions.  

Our paper is organized as follows. Section \ref{sec2} provides an overview of H\"ormander vector fields, associated Sobolev spaces and Sobolev embeddings. In Section \ref{sec3}, we delve into the study of the nonlinear eigenvalue problem \eqref{Dirichlet p-sub-Laplacian type} concerning smooth H\"ormander vector fields. Our main results encompass the existence of the first eigenvalue and  positiveness of the corresponding eigenfunction. The positiveness result is derived from the application of the Harnack inequality. Similarly to the classical case, determination of the first eigenvalue involves minimizing the Rayleigh quotient, which in turn allows us to find the best constant of the $L^{p}$-Poincar\'e-Friedrichs inequality for H\"ormander vector fields. Section \ref{sec4} focuses on investigating the simplicity and isolated nature of the first eigenvalue.

\section{Preliminaries}\label{sec2}
Unless otherwise stated, let $p>1$ and $m,n\in \mathbb{N}$ with $m\leq n$, where 
$\mathbb{N}=\{1,2,3, \ldots\}$.  If $A \subset B \subset \mathbb{R}^n$ such that  $\bar{A} \subset B$ and $\bar{A}$ is compact, then we write $A \Subset B$. If a normed space $X$ is continuously embedded in another normed space $Y$, then we write $X \hookrightarrow Y$. If $X$ is compactly embedded in $Y$, then we write $X \hookrightarrow \hookrightarrow Y$.

Let $U$ be a bounded domain in $\mathbb{R}^{n}$  with $n\geq 2$ and let 
\begin{equation*}
X_{i}=\sum_{k=1}^{n} b_{ik}(x) \partial_{x_{k}}, \quad i=1,\ldots,m,
\end{equation*}
be a collection of smooth vector fields defined on $U$. We adopt the notation 
$$X_{i}I(x):=\left(b_{i1}(x),\ldots,b_{in}(x)\right)^{\top}, \enspace x\in U$$ 
for $i=1,\ldots,m$. Given a multi-index $J=\left(j_1, \ldots, j_s\right) \in\{1, \ldots, m\}^s$ with $s\in \mathbb{N}$, we set
\begin{equation*}
X_{J}:=\left[X_{j_1}, \ldots\left[X_{j_{s-1}}, X_{j_s}\right] \ldots\right].
\end{equation*}
We call $X_{J}$  a commutator  of   length  $s$. If there exists a smallest $s\in \mathbb{N}$ such that the vectors
$\left\{X_{J}I(x)\right\}_{|J| \leq s}$ span $\mathbb{R}^{n}$ at each $x\in U$, then we say that   $X_{1}, \ldots, X_{m}$ satisfy the H\"ormander condition at step $s$. 

Throughout the paper we assume that   $X_{1}, \ldots, X_{m}$ satisfy the H\"ormander condition at step $s$. In the literature,  $X_{1}, \ldots, X_{m}$  are called the H\"ormander vector fields.

Now let $\Omega \Subset U$ be an open connected subset.  Let us recall the notion of control distance associated with  $X_{1}, \ldots, X_{m}$.
\begin{definition}
Given $\delta>0$, let $C_{1}(\delta)$ denote the class of absolutely continuous curves $\varphi:[0,1] \rightarrow \Omega$ satisfying
$$
\varphi^{\prime}(t)=\sum_{i=1}^{m} a_{i}(t)X_{i}\left({\varphi(t)}\right)\quad  \text {a.e. in } [0,1],
$$
where $a_{i}$ are measurable functions with upper bounds
$$
\left|a_{i}(\cdot)\right| \leq \delta \quad \text {a.e. in } [0,1] 
$$
for $i=1, \ldots, m$. Then for any two points $x,y\in \Omega$ we define the control distance by
\begin{equation*}
d_{X}(x, y):=\inf \left\{\delta>0:\enspace \exists \varphi \in C_{1}(\delta) \text { with } \varphi(0)=x \text{ and } \varphi(1)=y\right\}.
\end{equation*}
\end{definition} 

\begin{remark}
For simplicity, we adopt the notation $d_{X}$, although it also depends on the domain $\Omega$,  see \cite[Remark 1.30]{bramanti2023hormander}.
\end{remark}
Since $d_{X}$ is finite in $\Omega$ (see e.g. \cite[Theorem 1.45]{bramanti2023hormander}), it follows that $(\Omega, d_{X})$ is a metric space, known as the Carnot-Carath\'eodory space. Therefore,  $d_{X}$ is also called the Carnot-Carath\'eodory metric. Then an open metric ball centred at $x\in\Omega$ with the radius $r>0$ in $(\Omega, d_{X})$ is defined as follows
$$
B_{X}(x, r):=\left\{y \in \Omega: \enspace d_{X}(x, y)<r\right\}.
$$

Now we introduce the local homogeneous dimension $Q$ which plays a  role to prove, for example, H\"older continuity of eigenfunctions, see Theorem \ref{Holder continuity}. Let $X^{(k)}$ denote the collection of all commutators of  $X_{1}, \ldots, X_{m}$ of length $k$, that is,
$$
X^{(k)}:=\left\{X_{J} :\enspace J=\left(j_1, \ldots, j_k\right)\in\{1, \ldots, m\}^{k}, \enspace |J|=k \right\},
$$
where $k\in \mathbb{N}$ with $k\leq s$. Among the components of $X^{(1)},\ldots,X^{(s)}$, we select ones which are necessary to span $\mathbb{R}^{n}$ and label them with $Y_{1},\ldots,Y_{l}$. If $Y_{i}\in X^{(k)}$, then we assign a formal degree $\operatorname{d}(Y_i)=k$ and set
$$
\operatorname{d}(I):=\sum_{k=1}^n \operatorname{d}\left(Y_{i_k}\right) \quad \text { and } \quad a_{I}(x):=\det\left(Y_{i_1}, \ldots, Y_{i_n}\right)(x)
$$
for  multi-indices $I=\left(i_1, \ldots, i_n\right) \in\{1, \ldots, l\}^{n}$. Finally, by introducing the Nagel-Stein-Wainger polynomial at $x\in\Omega$ (see \cite{NSW1985})
$$
\Lambda(x, r):=\sum_{I}\left|a_{I}(x)\right| r^{\operatorname{d}(I)} \quad \text{for } r>0,
$$
now we are ready to define the local homogeneous dimension.
\begin{definition}\cite[p. 1771]{capogna1993embedding}
Given $x\in\Omega$, let  
\begin{equation*}
Q(x):=\lim _{r \rightarrow 0+} \frac{\log \Lambda(x, r)}{\log r}.
\end{equation*}
The number $Q(x)$ is called the pointwise homogeneous dimension relative to $\Omega$. 
\end{definition}

\begin{definition}\cite[p. 1771]{capogna1993embedding}
Let 
\begin{equation*}
Q:=\sup _{x \in \Omega} Q(x).
\end{equation*}   
The number $Q$ is called the local homogeneous dimension relative to $\Omega$.
\end{definition}
\begin{remark}
It is clear that $n\leq Q(x)\leq Q$. Moreover, the local homogeneous dimension $Q$ coincides with the generalized M\'etivier index $\Tilde{\nu}$, see \cite[Proposition 2.2 and Definition 1.2]{Chen2019}.
\end{remark}

Next we  give the definition of the Sobolev space $\mathcal{W}_{X}^{1, p}(\Omega)$ associated with  $X_{1}, \ldots, X_{m}$. Let us start with the notion of weak derivatives with respect to H\"ormander vector fields. A function $f\in L^{1}_{loc}(\Omega)$ is differentiable in the weak sense with respect to $X_{i}$ if there exists a function $g\in L^{1}_{loc}(\Omega)$ such that
\begin{equation*}
\int_{\Omega} g(x)  \varphi(x) d x=\int_{\Omega} f(x) X_{i}^{*} \varphi(x) d x\quad \text{ for all } \varphi \in C_{0}^{\infty}(\Omega),
\end{equation*}
where $X_{i}^{*}$ is the formal adjoint of $X_{i}$, which is given by
$$
X_{i}^{*}\varphi(x)=-\sum_{k=1}^n \partial_{x_{k}}\left(b_{ik} \varphi(x)\right).
$$
For $f\in L^{1}_{loc}(\Omega)$ we denote $X_{i}f=g$. Then the Sobolev space associated with $X_{1}, \ldots, X_{m}$ is defined by
\begin{equation*}
\mathcal{W}_{X}^{1, p}(\Omega):=\left\{f \in L^{p}(\Omega):\enspace X_{i} f \in L^{p}(\Omega)\quad \text{for } i=1, \ldots, m \right\}
\end{equation*}
with the norm
\begin{equation}\label{natural norm}
\|u\|_{\mathcal{W}_{X}^{1, p}(\Omega)}=\left(\int_{\Omega}\left(|u|^p+|X u|^p\right) d x\right)^{\frac{1}{p}},
\end{equation}
where $X:=\left(X_{1}, \ldots, X_{m}\right)$ is called the horizontal gradient and its length is given by
$$|X f|=\left(\sum_{i=1}^{m}\left(X_{i} f\right)^{2}\right)^{\frac{1}{2}}.$$
Then we introduce the trace zero Sobolev space $\mathcal{W}_{X,0}^{1, p}(\Omega)$ which is the completion of $C_{0}^{\infty}(\Omega)$ in $\mathcal{W}_{X}^{1, p}(\Omega)$. 
Let us recall some known results concerning $\mathcal{W}_{X,0}^{1, p}(\Omega)$. Since $\mathcal{W}_{X,0}^{1, p}(\Omega)$ is a  reflexive Banach space (see e.g. \cite[Theorem 1]{xu1990subelliptic}), then every bounded sequence in $\mathcal{W}_{X,0}^{1, p}(\Omega)$ has a weakly convergent subsequence by the Eberlein-\v{S}mulian theorem. Hence, we need to characterize weak convergence in $\mathcal{W}_{X,0}^{1, p}(\Omega)$.  When $p=2$, we refer to \cite[Proposition 2.6]{bramanti2023hormander}. When $p\neq 2$, we have the following proposition.
\begin{proposition}\label{Weak convergence characterization}
A sequence $v_n$ converges weakly to $v$ in $\mathcal{W}_{X,0}^{1, p}(\Omega)$ if and only if there exist $g_{i}\in L^{p}(\Omega)$ such that 
$$v_n\rightharpoonup v \enspace\text{ weakly in } L^{p}(\Omega) \quad \text{ and } \quad X_{i}v_n\rightharpoonup g_{i}\enspace \text{ weakly in } L^{p}(\Omega)$$ 
as $n\to\infty$, $i=1,\ldots,m$. In this case, $g_{i}=X_{i}v$.
\end{proposition}
\begin{proof}
The proof is standard, we refer to \cite[Proposition 1.8]{LeDret2018}, see also \cite[Corollary 11.70]{Leoni2017Sobolev} and \cite[Proposition 2.2]{capogna2024asymptotic}.
\end{proof}

\begin{theorem}\cite[Theorem 2.5]{capogna2024asymptotic}  (Poincar\'e-Friedrichs inequality). There exists a  constant $C>0$ such that 
\begin{equation}\label{Poincare inequality}
\int_{\Omega}|u|^{p} d x\leq C \int_{\Omega}|X u|^{p} d x \qquad \text{for all }u\in \mathcal{W}_{X,0}^{1, p}(\Omega).
\end{equation}
Here the constant $C$ is independent of $u$. See also \cite[Proposition 2.6]{chen2024sharp}.
\end{theorem}
In view of the Poincar\'e-Friedrichs inequality \eqref{Poincare inequality}, the space $\mathcal{W}_{X,0}^{1, p}(\Omega)$ can be equipped with the equivalent norm
\begin{equation}\label{Poincare norm}
\|u\|_{\mathcal{W}_{X,0}^{1, p}(\Omega)}:=\left(\int_{\Omega}\left|X u\right|^{p} d x\right)^{\frac{1}{p}}.
\end{equation}
Unless otherwise stated, we use the norm \eqref{Poincare norm} for $\mathcal{W}_{X,0}^{1, p}(\Omega)$. 

\begin{theorem}\label{compact Sobolev embedding}\cite[Corollary 3.3]{Danielli1991}
Let $1\leq p<\infty$. Then  $\mathcal{W}_{X,0}^{1, p}(\Omega)\hookrightarrow\hookrightarrow L^{p}(\Omega)$.
\end{theorem}

\begin{theorem}\label{continuous Sobolev embedding for 1<p<Q}

Let  $1< p <Q$. Then  $\mathcal{W}_{X,0}^{1, p}(\Omega)\hookrightarrow L^{q}(\Omega)$   for every $1\leq q \leq \frac{Qp}{Q-p}.$  
\end{theorem}
\begin{proof}
\cite[Theorem 2.3]{capogna1993embedding} and partition of the unity yield the result. 
\end{proof}
\begin{definition}
Let $0<\alpha<1$. A function $f:\Omega\to \mathbb{R}$ is said to be $\alpha$-H\"older continuous with respect to the control distance $d_{X}$, if
\begin{equation*}
\sup_{\stackrel{x,y\in \Omega}{x\neq y}}\frac{|f(x)-f(y)|}{d_{X}(x,y)^{\alpha}}<\infty.   
\end{equation*}
Such functions form the H\"older class $\mathcal{C}^{0,\alpha}_{X}(\Omega)$ associated with $X_{1}, \ldots, X_{m}$.
\end{definition}

\begin{theorem}\label{continuous Sobolev embedding for p>Q}
Let $p\geq Q$. Then $\mathcal{W}_{X,0}^{1, p}(\Omega)\hookrightarrow L^{q}(\Omega)$   for every $1\leq q < \infty$  
\end{theorem}
\begin{proof}
Let $p=Q$. Then  $\mathcal{W}_{X,0}^{1, p}(\Omega)\hookrightarrow L^{q}(\Omega)$   for every $1\leq q < \infty$ by \cite[Remark 1.1]{chen2024sharp}. Now let $p>Q$. By \cite[Remark 1.1]{chen2024sharp} we have $\mathcal{W}_{X,0}^{1, p}(\Omega)\hookrightarrow \mathcal{C}^{0, \alpha}_{X}(\overline{\Omega})$ provided $0<\alpha<1$. By \cite[Proposition 2.14]{bramanti2023hormander} we have $\mathcal{C}^{0, \alpha}_{X}(\overline{\Omega})\hookrightarrow C^{0,\alpha/s}(\overline{\Omega})$. This completes the proof.
\end{proof}

\section{Nonlinear spectral problem}\label{sec3}
We consider the nonlinear eigenvalue problem
\begin{equation}\label{eigenvalue problem}
\begin{aligned}
\sum_{i=1}^{m}X_{i}^{*}\left(\left|X u\right|^{p-2} X_{i}u\right)&=\lambda|u|^{p-2} u  &&\text { in } \Omega, \\
u&=0
&&\text { on } \partial \Omega,
\end{aligned}
\end{equation}
in the weak  sense.

\begin{definition}\label{def:eigenfunction}
A  function $u \in \mathcal{W}_{X,0}^{1, p}(\Omega)\setminus \{0\}$ is called an eigenfunction of  \eqref{eigenvalue problem} if
\begin{equation}\label{eigenfunction}
\int_{\Omega}|X u|^{p-2} Xu \cdot X \varphi d x=\lambda \int_{\Omega}|u|^{p-2} u \varphi d x
\end{equation}
holds for all $\varphi\in \mathcal{W}_{X,0}^{1, p}(\Omega)$. Such $\lambda$ is called an eigenvalue.
\end{definition}
Now we estimate a lower bound of eigenvalues.
\begin{proposition} \label{lower bound lemma}
Let $(u, \lambda)$ be an eigenpair. Then $\lambda$ has a lower bound
\begin{equation*}
\lambda \geq  \frac{1}{C|\Omega|^{1-\frac{p}{p^{*}}}}, 
\end{equation*}
where $|\Omega|$ denotes $n$-dimensional Lebesgue measure of $\Omega$, the number $p^{*}$ is given by  
\begin{equation}\label{p asterisk}
p^{*}=\begin{cases}
\frac{Qp}{Q-p}& \text{ if } 1<p<Q,\\
2p& \text{ if } p\geq Q,
\end{cases}
\end{equation}
and $C>0$ does not depend on  $u$ and  $\lambda$.
\end{proposition}
\begin{proof}
Let $(u,\lambda)$ be an eigenpair of  \eqref{eigenvalue problem}. The following  embedding
\begin{equation}\label{embeddings for lower bound of eigenvalues}
\mathcal{W}_{X,0}^{1, p}(\Omega)\hookrightarrow\begin{cases}
L^{\frac{Qp}{Q-p}}(\Omega) & \text{ if } 1<p<Q,\\
L^{2p}(\Omega) & \text{ if } p\geq Q,
\end{cases}
\end{equation}
follows from Theorem \ref{continuous Sobolev embedding for 1<p<Q} and Theorem \ref{continuous Sobolev embedding for p>Q}. Then combining the H\"older inequality and  Sobolev embedding \eqref{embeddings for lower bound of eigenvalues}, we obtain
\begin{equation*}
\int_{\Omega}|X u|^{p} d x= \lambda\int_{\Omega}|u|^{p} d x\leq \lambda|\Omega|^{1-\frac{p}{p^{*}}}\left(\int_{\Omega}|u|^{p^{*}}dx\right)^{\frac{p}{p^{*}}}\leq C\lambda|\Omega|^{1-\frac{p}{p^{*}}}\int_{\Omega}|X u|^{p} d x ,
\end{equation*}
where $p^{*}$ is given by \eqref{p asterisk}.
\end{proof}
\begin{corollary}
All eigenvalues are  positive.    
\end{corollary}
Although the next theorem will not be used in the sequel, it has an independent interest by showing that the spectrum is a closed set in $\mathbb{R}$.
\begin{theorem}\label{closed set}
The set of eigenvalues is closed.
\end{theorem}
\begin{proof} 
Let $\{\Tilde{\lambda}_{j}\}$ be a  sequence of eigenvalues which converges to some $\Tilde{\lambda}\in \mathbb{R}$.  We denote an eigenfunction corresponding to $\Tilde{\lambda}_{j}$  by $\Tilde{u}_{j}$. Note that if $u$ is an eigenfunction, then so is $\frac{1}{\|u\|_{p}}u$. Hence, without loss of generality, we assume that
$\left\|\Tilde{u}_{j}\right\|_{p}=1$. Therefore,
$$
\Tilde{\lambda}_{j}=\int_{\Omega}\left|X\Tilde{u}_{j}\right|^p d x,
$$
which implies that $\{\Tilde{u}_{j}\}$ is bounded in $\mathcal{W}_{X,0}^{1, p}(\Omega)$. Therefore, by the Eberlein-\v{S}mulian theorem  and Theorem \ref{compact Sobolev embedding}, there exist a subsequence $\{\Tilde{u}_{j_{k}}\}$ and a function $\Tilde{u}\in \mathcal{W}_{X,0}^{1, p}(\Omega)$ such that
\begin{equation}\label{strong and weak convergence}
\Tilde{u}_{j_{k}} \to \Tilde{u}\enspace \text{ strongly in } L^{p}(\Omega)\quad\text{and} \quad X \Tilde{u}_{j_{k}} \rightharpoonup  X \Tilde{u} \enspace \text{ weakly in } L^{p}(\Omega). 
\end{equation}
Regarding $\Tilde{u}_{j_{k}}-\Tilde{u}$ as a test function, we insert it into \eqref{eigenfunction} to get
\begin{equation*}
\begin{aligned}
&\int_{\Omega}\left(\left|X \Tilde{u}_{j_{k}}\right|^{p-2} X \Tilde{u}_{j_{k}}-|X \Tilde{u}|^{p-2} X \Tilde{u}\right) \cdot \left(X\Tilde{u}_{j_{k}}-X \Tilde{u}\right) d x \\
&=\Tilde{\lambda}_{j_{k}} \int_{\Omega}\left|\Tilde{u}_{j_{k}}\right|^{p-2} \Tilde{u}_{j_{k}}\left(\Tilde{u}_{j_{k}}-\Tilde{u}\right) d x-\int_{\Omega}|X \Tilde{u}|^{p-2} X \Tilde{u} \cdot\left(X \Tilde{u}_{j_{k}}-X \Tilde{u}\right) d x .
\end{aligned}
\end{equation*}
Taking into account \eqref{strong and weak convergence} and letting $k\to \infty$, the above equality yields 
\begin{equation}\label{closed set: limit}
\lim _{k \rightarrow \infty} \int_{\Omega}\left(\left|X \Tilde{u}_{j_{k}}\right|^{p-2} X \Tilde{u}_{j_{k}}-|X \Tilde{u}|^{p-2} X \Tilde{u}\right) \cdot\left(X \Tilde{u}_{j_{k}}-X \Tilde{u}\right) d x=0 .
\end{equation}
If we prove the following convergence
\begin{equation}\label{closed set: strong convergence of the horizontal gradient}
X \Tilde{u}_{j_{k}} \to  X \Tilde{u}\quad \text{ strongly in } L^{p}(\Omega), 
\end{equation}
then $\Tilde{u}$ is an eigenfunction corresponding to $\Tilde{\lambda}$ and the proof is complete. 

In order to prove convergence \eqref{closed set: strong convergence of the horizontal gradient}, we treat $1<p<2$ and $p\geq2$ cases separately.

Let $p\geq 2$. A modification  of the inequality \eqref{convexity (a)} yields
\begin{equation}\label{closed set: inequality}
C\left|\omega_{2}-\omega_{1}\right|^{p} \leq p\left(\left|\omega_{2}\right|^{p-2} \omega_{2}-\left|\omega_{1}\right|^{p-2} \omega_{1}\right) \cdot\left(\omega_{2}-\omega_{1}\right) \quad \text{ for all }\omega_{1}, \omega_{2} \in \mathbb{R}^{m}.
\end{equation}
Inserting $\omega_{1}=X \Tilde{u}$ and $\omega_{2}=X \Tilde{u}_{j_{k}}$ into \eqref{closed set: inequality} and integrating over $\Omega$, we see that \eqref{closed set: limit} implies the strong convergence of $\{X \Tilde{u}_{j_{k}}\}$ in $L^{p}$-norm 
$$
\lim _{k \rightarrow \infty} \int_{\Omega}\left|X \Tilde{u}_{j_{k}}-X \Tilde{u}\right|^{p} d x=0 .
$$

Let $1<p<2$.
A modification  of the inequality \eqref{convexity (b)} leads to
\begin{equation}\label{closed set: inequality 1<p<2}
C\frac{|\omega_{2}-\omega_{1}|^{2}}{(|\omega_{2}|+|\omega_{1}|)^{2-p}}   \leq p\left(\left|\omega_{2}\right|^{p-2} \omega_{2}-\left|\omega_{1}\right|^{p-2} \omega_{1}\right)\cdot\left(\omega_{2}-\omega_{1}\right) \quad \text{ for all }\omega_{1}, \omega_{2} \in \mathbb{R}^{m}.
\end{equation}
Inserting $\omega_{1}=X \Tilde{u}$ and $\omega_{2}=X \Tilde{u}_{j_{k}}$ into \eqref{closed set: inequality 1<p<2} and integrating over $\Omega$, we have
\begin{equation*}
\lim_{k\to\infty}\int_{\Omega}\frac{|X \Tilde{u}_{j_{k}}-X \Tilde{u}|^{2}}{(|X \Tilde{u}_{j_{k}}|+|X \Tilde{u}|)^{2-p}}dx=0.
\end{equation*}
To get the strong convergence of $\{X \Tilde{u}_{j_{k}}\}$ in $L^{p}$-norm for $1<p<2$, we apply the H\"older inequality 
\begin{equation}\label{1<p<2 convergence}
\begin{aligned}
&\int_{\Omega}|X \Tilde{u}_{j_{k}}-X \Tilde{u}|^{p}dx=\int_{\Omega}(|X \Tilde{u}_{j_{k}}|+|X \Tilde{u}|\big)^{\frac{p(2-p)}{2}}\frac{|X \Tilde{u}_{j_{k}}-X \Tilde{u}|^{p}}{(|X \Tilde{u}_{j_{k}}|+|X \Tilde{u}|)^{\frac{p(2-p)}{2}}}dx
\\
&\leq\left(\int_{\Omega}\Big(|X \Tilde{u}_{j_{k}}|+|X \Tilde{u}|\Big)^{p}dx \right)^{\frac{2-p}{2}}\left(\int_{\Omega}\frac{|X \Tilde{u}_{j_{k}}-X \Tilde{u}|^{2}}{(|X \Tilde{u}_{j_{k}}|+|X \Tilde{u}|)^{2-p}}dx \right)^{\frac{p}{2}}
\\
&\leq C\left(\int_{\Omega}\frac{|X \Tilde{u}_{j_{k}}-X \Tilde{u}|^{2}}{(|X \Tilde{u}_{j_{k}}|+|X \Tilde{u}|)^{2-p}}dx \right)^{\frac{p}{2}}
\end{aligned}
\end{equation}
Letting $k\to\infty$, we obtain convergence \eqref{closed set: strong convergence of the horizontal gradient}. This completes the proof.
\end{proof}
The following lemma will be used to prove Theorem \ref{Holder continuity} and Theorem \ref{simplicity theorem}.
\begin{lemma}\label{L infty boundedness}
Let $(u,\lambda)$ be an eigenpair. Then $u\in L^{\infty}(\Omega)$.
\end{lemma}
\begin{proof}
The case $p>Q$ follows from the  embedding $\mathcal{W}_{X,0}^{1, p}(\Omega)\hookrightarrow \mathcal{C}^{0, \alpha}_{X}(\overline{\Omega})$ provided $0<\alpha<1$, see e.g. \cite[Remark 1.1]{chen2024sharp}.

Let $1<p\leq Q$ and $k>0$. First we will prove that $\esssup_{\Omega}u<\infty$. Assume that 
$$\left|\{x\in\Omega:\enspace u(x)>0\}\right|>0.$$
Otherwise, it suffices to show only $\essinf_{\Omega}u>-\infty$.

Let us define the function
\begin{equation*}
\varphi(x):=\max \{u(x)-k, 0\},
\end{equation*}
which is admissible in $\mathcal{W}_{X,0}^{1, p}(\Omega)$. Indeed, since
\begin{equation*}
    \int_{\Omega}\varphi^{p}dx=\int_{\Omega_{k}}(u-k)^{p}dx\leq \int_{\Omega_{k}}u^{p}dx\leq \int_{\Omega}|u|^{p}dx<\infty,
\end{equation*} 
then $\varphi\in L^{p}(\Omega)$, where
\begin{equation*}
\Omega_{k}:=\{x \in \Omega : \enspace u(x)>k\} .
\end{equation*}
Moreover, 
\begin{equation*}
\int_{\Omega}|X\varphi|^{p}dx=\int_{\Omega_{k}}|Xu|^{p}dx\leq \int_{\Omega}|Xu|^{p}dx<\infty.
\end{equation*}
Finally, $\varphi|_{\partial\Omega}=\max\{u|_{\partial\Omega}-k,0\}=0$, thus $\varphi$ can be regarded as a test function  so that
\begin{equation*}
\int_{\Omega_{k}}|X u|^{p} d x=\lambda \int_{\Omega_{k}}u^{p-1}(u-k) d x    
\end{equation*}
holds. We have $\mathcal{W}_{X,0}^{1, p}(\Omega)\hookrightarrow L^{1}(\Omega)$ for all $1<p\leq Q$, that is, the case $1<p<Q$ follows from Theorem \ref{continuous Sobolev embedding for 1<p<Q}, the case $p=Q$ follows from Theorem \ref{continuous Sobolev embedding for p>Q}. Therefore, $u\in L^{1}(\Omega)$. Since
\begin{equation}\label{upper bound: some properties 0}
 k\left|\Omega_{k}\right|=   \int_{\Omega_{k}}kdx\leq \int_{\Omega_{k}} u dx\leq \int_{\Omega} |u| dx=\|u\|_{1},
\end{equation}
it follows that
\begin{equation}\label{upper bound: some properties}
 \lim_{k\to \infty}\left|\Omega_{k}\right|=0.
\end{equation}
Multiplying the simple inequality
\begin{equation}
\begin{aligned}
u^{p-1}   \leq 2^{p-1}(u-k)^{p-1}+2^{p-1} k^{p-1} \enspace \text{in } \Omega_{k}
\end{aligned}
\end{equation}
by $u-k$ and integrating over $\Omega_{k}$ we obtain
\begin{equation}\label{upper bound: Clarkson's inequality}
\int_{\Omega_{k}} u^{p-1}(u-k) d x \leq 2^{p-1} \int_{\Omega_{k}}(u-k)^{p} d x+2^{p-1} k^{p-1} \int_{\Omega_{k}}(u-k) d x.
\end{equation}
Then applying the H\"older inequality and  Sobolev embedding \eqref{embeddings for lower bound of eigenvalues}, we obtain
\begin{equation}\label{upper bound: Poincare and lower bound}
\begin{aligned}
\int_{\Omega_{k}}(u-k)^{p} d x &\leq|\Omega_{k}|^{1-\frac{p}{p^{*}}}\left(\int_{\Omega_{k}}(u-k)^{p^{*}}dx\right)^{\frac{p}{p^{*}}}= |\Omega_{k}|^{1-\frac{p}{p^{*}}}\left(\int_{\Omega}\varphi^{p^{*}}dx\right)^{\frac{p}{p^{*}}}\\
&\leq C \left|\Omega_{k}\right|^{1-\frac{p}{p^{*}}} \int_{\Omega}|X \varphi|^{p} d x=C\left|\Omega_{k}\right|^{1-\frac{p}{p^{*}}} \int_{\Omega_{k}}|X u|^{p} d x\\
&=\lambda  C \left|\Omega_{k}\right|^{1-\frac{p}{p^{*}}}\int_{\Omega_{k}}|u|^{p-2}u (u-k)dx,
\end{aligned}
\end{equation}
where $p^{*}$ is expressed as \eqref{p asterisk}. Since $\lambda$ is positive, we multiply both sides of \eqref{upper bound: Clarkson's inequality} by $\lambda  C \left|\Omega_{k}\right|^{1-\frac{p}{p^{*}}}$ and use the inequality \eqref{upper bound: Poincare and lower bound} to obtain
\begin{equation*}
\left(1-\lambda C \left|\Omega_{k}\right|^{1-\frac{p}{p^{*}}}2^{p-1}\right) \int_{\Omega_{k}}(u-k)^{p} d x \leq \lambda C\left|\Omega_{k}\right|^{1-\frac{p}{p^{*}}}2^{p-1}k^{p-1}  \int_{\Omega_{k}}(u-k) d x .
\end{equation*}
Let us show that the number $\lambda C \left|\Omega_{k}\right|^{1-\frac{p}{p^{*}}} 2^{p-1}$ can be arbitrarily small for large $k$. In view of \eqref{upper bound: some properties 0}, we have
\begin{equation*}
\lambda C \left|\Omega_{k}\right|^{1-\frac{p}{p^{*}}} 2^{p-1} \leq  \lambda C \frac{\|u\|_{1}^{1-\frac{p}{p^{*}}}}{k^{1-\frac{p}{p^{*}}}}2^{p-1}.
\end{equation*}
Since $\|u\|_{1}/k$ tends to $0$, then  for $\varepsilon>0$ there exists $k_{\varepsilon}\in \mathbb{N}$ such that for all $k\geq k_{\varepsilon}$
\begin{equation}\label{upped bound: some convergence}
\lambda C \frac{\|u\|_{1}^{1-\frac{p}{p^{*}}}}{k^{1-\frac{p}{p^{*}}}}2^{p-1}  < \varepsilon
\end{equation}
holds. Here 
\begin{equation*}
 \lambda^{{\frac{p^{*}}{ p^{*}-p}}}C^{{\frac{p^{*}}{ p^{*}-p}}}2^{\frac{(p-1)p^{*}}{ p^{*}-p}}\|u\|_{1} \varepsilon^{-\frac{p^{*}}{ p^{*}-p}}< k_{\varepsilon}.
\end{equation*}
We can take 
$$k_{\varepsilon}= \lambda^{{\frac{p^{*}}{ p^{*}-p}}}C^{{\frac{p^{*}}{ p^{*}-p}}}2^{\frac{(p-1)p^{*}}{ p^{*}-p}} \|u\|_{1} \varepsilon^{-\frac{p^{*}}{ p^{*}-p}}+1$$ 
for \eqref{upped bound: some convergence}. Now let $\varepsilon=\frac{1}{2}$.  Then 
\begin{equation*}
 \lambda C \left|\Omega_{k}\right|^{1-\frac{p}{p^{*}}}2^{p-1} < \frac{1}{2}    
\end{equation*}
holds for all $k \geq k_{\varepsilon}$. Therefore,
\begin{equation}\label{upper bound: Lp-L1 inequality}
\int_{\Omega_{k}}(u-k)^{p} d x <\lambda C\left|\Omega_{k}\right|^{1-\frac{p}{p^{*}}}2^{p} k^{p-1}  \int_{\Omega_{k}}(u-k) d x  \quad \text{ for all } k \geq k_{\varepsilon}.
\end{equation}
By the H\"older inequality, we have
\begin{equation}\label{upper bound: Holder inequality}
\left(\int_{\Omega_{k}}(u-k) d x \right)^{p}\leq |\Omega_{k}|^{p-1}   \int_{\Omega_{k}}(u-k)^{p} d x.
\end{equation}
Finally, combining \eqref{upper bound: Lp-L1 inequality} and \eqref{upper bound: Holder inequality}, we arrive at
\begin{equation}\label{upper bound: desired inequality}
\int_{\Omega_{k}}(u-k) d x <  (2^{p}\lambda C)^{\frac{1}{p-1}} k\left|\Omega_{k}\right|^{1+\frac{p^{*}-p}{p^{*}(p-1)}} \quad \text{ for all } k \geq k_{\varepsilon}.  
\end{equation}
This is the desired inequality to bound $\esssup_{\Omega}u$ by  \cite[Lemma 5.1, p. 71]{Ladyzhenskaya1968}. One can prove that $\essinf_{\Omega}u>-\infty$ by replacing $u$ with $-u$.
\end{proof}

\begin{theorem}\label{Holder continuity}
All eigenfunctions  are H\"older continuous.
\end{theorem}
\begin{proof}
Let $u$ be an eigenfunction. The proof is split in two cases.

Let $1<p\leq Q$. Our eigenvalue problem \eqref{eigenvalue problem} can be seen as a particular case of the general subelliptic equation with $f=\lambda|u|^{p-2} u $ presented in \cite[p. 1767]{capogna1993embedding}. Since $u$ is essentially bounded in $\Omega$ by Lemma \ref{L infty boundedness}, then there exists $0<\alpha<1$ such that
$$
\underset{x, y \in \Omega}{\operatorname{ess} \sup } \frac{|u(x)-u(y)|}{d_{X}(x,y)^\alpha} < \infty
$$
by \cite[Theorem 3.35]{capogna1993embedding}. After a redefinition in a set of Lebesgue measure zero (including the boundary of $\Omega$), we derive the H\"older continuity of $u$ with respect to  $d_{X}$, so $u\in \mathcal{C}^{0,\alpha}_{X}(\overline{\Omega})$.

Let $p>Q$. Then $u$ is H\"older continuous with respect to  $d_{X}$  by \cite[Remark 1.1]{chen2024sharp}.
\end{proof}

Given $u \in \mathcal{W}_{X,0}^{1, p}(\Omega) \setminus\{0\}$, we define the Rayleigh quotient by
\begin{equation}\label{Rayleigh quotient}
\frac{\int_{\Omega}|X u|^p d x}{\int_{\Omega}|u|^p d x}.
\end{equation}
We will show  that minimization  the Rayleigh quotient gives the first eigenvalue of \eqref{eigenvalue problem}. 
\begin{theorem}\label{existence and uniqueness of the Rayleigh quotient}
There exists    $u_{1}\in \mathcal{W}_{X,0}^{1, p}(\Omega)\setminus\{0\}$   such that
$$\frac{\int_{\Omega}|X u_{1}|^p d x}{\int_{\Omega}|u_{1}|^p d x} =\inf_{u \in \mathcal{W}_{X,0}^{1, p}(\Omega)\setminus\{0\}}\frac{\int_{\Omega}|X u|^p d x}{\int_{\Omega}|u|^p d x} =\lambda_{1}.$$
Moreover, $(\lambda_{1}, u_{1})$ satisfies \eqref{eigenfunction} for all  $\varphi\in \mathcal{W}_{X,0}^{1, p}(\Omega)$ and $\lambda_{1}$ is the smallest eigenvalue. We call $(\lambda_{1}, u_{1})$ the first eigenpair.
\end{theorem}
\begin{proof}
If $u$ is an eigenfunction of \eqref{eigenvalue problem}, then so is $Cu$ with $C\neq 0$.
Thus, without loss of generality, it suffices to minimize $\int_{\Omega}|X u|^p d x$ with  the normalization $\|u\|_{p}=1$.

First we prove existence of a minimizer. Let us define the $p$-energy functional $E:\mathcal{W}_{X,0}^{1, p}(\Omega)\to \mathbb{R}$ by
\begin{equation}\label{p-energy functional}
E(u):=\int_{\Omega}|X u|^{p} d x.
\end{equation}
Also let
\begin{equation*}
E_{0}:=\inf_{\stackrel{u \in \mathcal{W}_{X,0}^{1, p}(\Omega)}{\|u\|_{p}=1}} \int_{\Omega}|X u|^{p} d x.
\end{equation*}
We choose a minimizing sequence  $\left\{\Tilde{u}_{j}\right\}$ in $ \mathcal{W}_{X,0}^{1, p}(\Omega)$ with $\|\Tilde{u}_{j}\|_{p}=1$ such that
\begin{equation*}
\int_{\Omega}\left|X\Tilde{u}_{j}\right|^{p} d x<E_{0}+\frac{1}{j}.
\end{equation*}
Therefore, $\left\{\Tilde{u}_{j}\right\}$ is bounded in
$\mathcal{W}_{X,0}^{1, p}(\Omega)$.  By the Eberlein-\v{S}mulian theorem,  there exist  a subsequence $\{\Tilde{u}_{j_{k}}\}$ and a function $u_{1} \in \mathcal{W}_{X,0}^{1, p}(\Omega)$ such that
\begin{equation*}
\Tilde{u}_{j_{k}} \rightharpoonup  u_{1}\enspace \text{ weakly in } \mathcal{W}_{X,0}^{1, p}(\Omega). 
\end{equation*}
 Now we show that $u_{1}$ is a minimizer. Since $X\Tilde{u}_{j_{k}} \rightharpoonup  Xu_{1}$ weakly in $L^{p}(\Omega)$ by Proposition \ref{Weak convergence characterization}, then in view of the weakly lower semicontinuity of $\|\cdot\|_{p}$ norm we have
\begin{equation*}
E(u_{1}) \leq \liminf _{j_{k} \rightarrow \infty} E\left(\Tilde{u}_{j_{k}}\right)=E_{0}. 
\end{equation*}

Let $\lambda_1:=\int_{\Omega}|X u_1|^{p} d x$. Now we will show that  $\lambda_{1}$ is an eigenvalue and $u_{1}$ is the corresponding eigenfunction. For every $\varphi\in \mathcal{W}_{X,0}^{1, p}(\Omega)$  we define the functional  $f:\mathbb{R}\to \mathbb{R}$ by
\begin{equation*}
f(\varepsilon):=\frac{\int_{\Omega}|X(u_{1}+\varepsilon \varphi)|^{p} d x}{\int_{\Omega}|u_{1}+\varepsilon \varphi|^{p} d x}.
\end{equation*}
Since $u_{1}$ is the minimizer,
it follows that  $f^{\prime}(0)=0$, that is,
\begin{equation*}
\int_{\Omega}|u_{1}|^{p} d x \int_{\Omega}|X u_{1}|^{p-2} X u_{1} \cdot X \varphi d x=\int_{\Omega}|u_{1}|^{p-2} u_{1} \varphi d x \int_{\Omega}|X u_{1}|^{p} d x.
\end{equation*}
Therefore, taking into account $\|u_{1}\|_{p}=1$, we have
\begin{equation*}
\begin{aligned}
\int_{\Omega}|X u_{1}|^{p-2} X u_{1} \cdot X \varphi d x &= \int_{\Omega}|X u_{1}|^{p} d x \int_{\Omega}|u_{1}|^{p-2} u_{1} \varphi dx\\
&=\lambda_{1} \int_{\Omega}|u_{1}|^{p-2} u_{1} \varphi d x
\end{aligned}
\end{equation*}
for all $\varphi\in \mathcal{W}_{X,0}^{1, p}(\Omega)$, so $(\lambda_{1}, u_{1})$ is an eigenpair. It only remains to prove that $\lambda_{1}$ is the smallest. Indeed, every eigenvalue $\Tilde{\lambda}\neq\lambda_{1}$ with a corresponding eigenfunction $\Tilde{u}$ with $\|\Tilde{u}\|_{p}=1$ satisfies
\begin{equation*}
\lambda_{1}=\int_{\Omega}|X u_{1}|^p d x\leq \int_{\Omega}|X \Tilde{u}|^p d x=\Tilde{\lambda} \int_{\Omega} |\Tilde{u}|^{p}d x=\Tilde{\lambda},
\end{equation*}
therefore, $\lambda_{1}$ is the smallest eigenvalue.
\end{proof}
As a consequence of Theorem \ref{existence and uniqueness of the Rayleigh quotient}, we have found the best constant in \eqref{Poincare inequality}.
\begin{corollary}\label{the best constant of Poincare}
The Poincar\'e-Friedrichs inequality for H\"ormander vector fields
\begin{equation*}
\int_{\Omega}|u|^{p} d x\leq \lambda_{1}^{-1} \int_{\Omega}|X u|^{p} d x 
\end{equation*}
holds for all $u\in \mathcal{W}_{X,0}^{1, p}(\Omega)$. The constant $\lambda_{1}^{-1}$ is sharp.
\end{corollary}

Moreover, we can obtain the nonnegativity of $u_{1}$.
\begin{corollary}\label{|u_1| minimizes}
Let $u_{1}$ be the minimizer in Theorem \ref{existence and uniqueness of the Rayleigh quotient}. Then $|u_{1}|$ also minimizes the Rayleigh quotient.
\end{corollary}
\begin{proof}
Proposition \ref{proposition Xu} ensures that $|u_{1}| \in \mathcal{W}_{X,0}^{1, p}(\Omega)$ and  $|Xu_{1}|=|X|u_{1}||$  a.e. in $\Omega$.
\end{proof}
Now we recall the Harnack inequality from \cite{Lu1996} adapted to the eigenvalue problem \eqref{eigenvalue problem}, which will be used to prove the positivity of the first eigenfunction $u_{1}$.
\begin{theorem}\label{Harnack inequality}
Let $x_{0}\in\Omega$ such that $B_{X}(x_0, 3r)\subset \Omega$ for some $r>0$. Suppose $u$ is a nonnegative eigenfunction of \eqref{eigenvalue problem} in  $B_{X}(x_0, 3r)$. Then
\begin{equation}\label{Harnack}
\esssup _{B_{X}(x_0, r)} u(x) \leq C \essinf _{B_{X}(x_0, r)} u(x),
\end{equation}
where $C>0$ is a constant.
\end{theorem}
\begin{proof}
The proof can be found in \cite[Corollary 3.11]{Lu1996}. Note that since $b_0=0$ in \cite[(3.1) formula, p. 314]{Lu1996}, therefore the boundedness assumption in \cite[Corollary 3.11]{Lu1996} can be relaxed in our case, see \cite[p. 318]{Lu1996} for the details (cf \cite[p. 724]{Trudinger1967}).
\end{proof}
\begin{remark}
Harnack inequality \eqref{Harnack} for $1<p\leq Q$ can be also found in \cite[Theorem 3.1]{capogna1993embedding}. In \cite[Theorem 3.1]{capogna1993embedding}, the term $K(R)$ can be zero.
\end{remark}

\begin{corollary}\label{positivity of the first eigenfuncion}
We can choose the first eigenfunction $u_{1}$ to be positive in $\Omega$.
\end{corollary}
\begin{proof}
Since $|u_{1}|$ also minimizes \eqref{Rayleigh quotient} by Corollary \ref{|u_1| minimizes}, we can choose $u_{1}$ to be nonnegative in $\Omega$. Theorem \ref{Harnack inequality} asserts that 
$$u_{1}>0 \quad \text{in } \Omega$$ 
or 
$$u_{1} = 0 \quad \text{in } \Omega.$$  
The latter  contradicts Definition \ref{def:eigenfunction}, so  $u_{1}>0$ in $\Omega$. 
\end{proof}

\section{Simplicity and isolatedness of the first eigenvalue}\label{sec4}
It is well-known that the first eigenvalue of the (classic) Dirichlet $p$-Laplacian is simple and isolated, see \cite{Lin1990}. Likewise, it is reasonable to anticipate such results for the subelliptic $p$-Laplacian. Although in contrast to $C^{1,\alpha}$ regularity of eigenfunctions of the Dirichlet $p$-Laplacian (see e.g. \cite{dibenedetto1983}),  we were able to prove that  eigenfunctions of \eqref{eigenvalue problem}  are H\"older continuous with respect to the control distance. Nevertheless, the obtained results allow us to prove the simplicity and isolatedness of $\lambda_{1}$.
\begin{theorem}\label{simplicity theorem}
The first eigenvalue $\lambda_{1}$ is simple. In other words, if there exist two positive eigenfunctions $u_{1}$ and $\Tilde{u}_{1}$ associated with $\lambda_{1}$, then they are proportional.
\end{theorem}
\begin{proof}
The proof is an adaptation of the $p$-Laplacian case. Suppose there exist two distinct eigenfunctions $u_{1}$ and $\Tilde{u}_{1}$ corresponding to $\lambda_{1}$. By Corollary \ref{positivity of the first eigenfuncion}, it follows that $u_{1}>0$ and $\Tilde{u}_{1}>0$ in $\Omega$. The main idea of the proof is an appropriate choice of test functions which leads to the conclusion that the eigenfunctions are unique modulo scaling. Given $\varepsilon>0$, let
\begin{equation*}
\varphi:=\frac{(u_{1}+\varepsilon)^{p}-(\Tilde{u}_{1}+\varepsilon)^{p}}{(u_{1}+\varepsilon)^{p-1}} 
\end{equation*}
be a test function for the eigenfunction $u_{1}$, and let
\begin{equation*}
\psi:=\frac{(\Tilde{u}_{1}+\varepsilon)^{p}-(u_{1}+\varepsilon)^{p}}{(\Tilde{u}_{1}+\varepsilon)^{p-1}}
\end{equation*}
be a test function for the eigenfunction $\Tilde{u}_{1}$. Note that 
$$\sup_{\Omega}u_{1}<\infty \quad \text{and}\quad\sup_{\Omega}\Tilde{u}_{1}<\infty$$
by Lemma \ref{L infty boundedness}  and Theorem \ref{Holder continuity}, thus $\varphi$ and $\psi$ are well defined and belong to $\mathcal{W}_{X,0}^{1, p}(\Omega)$. Indeed, from $u_{1},\Tilde{u}_{1}\in\mathcal{C}^{0, \alpha}_{X}(\overline{\Omega})$ (see Theorem \ref{Holder continuity}), we have
$$
X \varphi=\left\{1+(p-1)\left(\frac{\Tilde{u}_{1}+\varepsilon}{u_{1}+\varepsilon}\right)^p\right\} X u_{1}-p\left(\frac{\Tilde{u}_{1}+\varepsilon}{u_{1}+\varepsilon}\right)^{p-1} X \Tilde{u}_{1}
$$
and
$$
X \psi=\left\{1+(p-1)\left(\frac{u_{1}+\varepsilon}{\Tilde{u}_{1}+\varepsilon}\right)^p\right\} X \Tilde{u}_{1}-p\left(\frac{u_{1}+\varepsilon}{\Tilde{u}_{1}+\varepsilon}\right)^{p-1} X u_{1},
$$
it follows that $u_{1}, \Tilde{u}_{1}\in \mathcal{W}_{X,0}^{1, p}(\Omega)$. Substituting  $u_{1},\varphi$ and $\Tilde{u}_{1},\psi$ into  \eqref{eigenfunction} respectively, and then combining them we derive
\begin{equation}\label{long calculation}
\begin{aligned}
&\lambda_{1}\int_{\Omega}\left(\left(\frac{u_{1}}{u_{1}+\varepsilon}\right)^{p-1}-\left(\frac{\Tilde{u}_{1}}{\Tilde{u}_{1}+\varepsilon}\right)^{p-1}\right)\big((u_{1}+\varepsilon)^{p}-(\Tilde{u}_{1}+\varepsilon)^{p}\big) d x\\
&= \int_{\Omega}\left(\left(1+(p-1)\left(\frac{\Tilde{u}_{1}+\varepsilon}{u_{1}+\varepsilon}\right)^{p}\right)\left|X u_{1}\right|^{p}+\left(1+(p-1)\left(\frac{u_{1}+\varepsilon}{\Tilde{u}_{1}+\varepsilon}\right)^{p}\right)\left|X \Tilde{u}_{1}\right|^{p}\right) d x \\
&-\int_{\Omega}\left(p\left(\frac{\Tilde{u}_{1}+\varepsilon}{u_{1}+\varepsilon}\right)^{p-1}\left|X u_{1}\right|^{p-2} X u_{1} \cdot X \Tilde{u}_{1}+p\left(\frac{u_{1}+\varepsilon}{\Tilde{u}_{1}+\varepsilon}\right)^{p-1}\left|X \Tilde{u}_{1}\right|^{p-2} X \Tilde{u}_{1} \cdot X u_{1}\right) d x \\
&= \int_{\Omega}\big((u_{1}+\varepsilon)^{p}-(\Tilde{u}_{1}+\varepsilon)^{p}\big)\big(\left|X \log (u_{1}+\varepsilon)\right|^{p}-\left|X \log (\Tilde{u}_{1}+\varepsilon)\right|^{p}\big) d x \\
&-\int_{\Omega} p (\Tilde{u}_{1}+\varepsilon)^{p}\left|X \log (u_{1}+\varepsilon)\right|^{p-2} X \log (u_{1}+\varepsilon) \cdot\big(X \log (\Tilde{u}_{1}+\varepsilon)-X \log (u_{1}+\varepsilon)\big) d x \\
&-\int_{\Omega} p (u_{1}+\varepsilon)^{p}\left|X \log (\Tilde{u}_{1}+\varepsilon)\right|^{p-2} X \log (\Tilde{u}_{1}+\varepsilon) \cdot\big(X \log (u_{1}+\varepsilon)-X \log (\Tilde{u}_{1}+\varepsilon)\big) d x.
\end{aligned}
\end{equation}

Let $p\geq 2$.  Multiplying the inequality \eqref{convexity (a)} with $\omega_{1}=X \log (\Tilde{u}_{1}+\varepsilon)$ and $\omega_{2}=X \log (u_{1}+\varepsilon)$ by $(u_{1}+\varepsilon)^p$ and then integrating over $\Omega$ we obtain
\begin{equation}\label{simplicity: convexity (a) vu}
\begin{aligned}
0  &\leq C \int_{\Omega}\frac{1}{(\Tilde{u}_{1}+\varepsilon)^{p}}\big|(\Tilde{u}_{1}+\varepsilon) X u_{1}-(u_{1}+\varepsilon) X \Tilde{u}_{1}\big|^{p} d x \\
& \leq \int_{\Omega}(u_{1}+\varepsilon)^{p}\big(\left|X \log (u_{1}+\varepsilon)\right|^{p}-\left|X \log (\Tilde{u}_{1}+\varepsilon)\right|^{p}\big) d x\\
&-\int_{\Omega} p (u_{1}+\varepsilon)^{p}\left|X \log (\Tilde{u}_{1}+\varepsilon)\right|^{p-2} X \log (\Tilde{u}_{1}+\varepsilon) \cdot\big(X \log (u_{1}+\varepsilon)-X \log (\Tilde{u}_{1}+\varepsilon)\big) d x.
\end{aligned}
\end{equation}
Multiplying the inequality \eqref{convexity (a)} with $\omega_{1}=X \log (u_{1}+\varepsilon)$ and $\omega_{2}=X \log (\Tilde{u}_{1}+\varepsilon)$ by $(\Tilde{u}_{1}+\varepsilon)^p$ and then integrating over $\Omega$ we obtain
\begin{equation}\label{simplicity: convexity (a) uv}
\begin{aligned}
0 & \leq C \int_{\Omega}\frac{1}{(u_{1}+\varepsilon)^{p}}\big|(\Tilde{u}_{1}+\varepsilon) X u_{1}-(u_{1}+\varepsilon) X \Tilde{u}_{1}\big|^{p} d x \\
& \leq-\int_{\Omega}(\Tilde{u}_{1}+\varepsilon)^{p}\big(\left|X \log (u_{1}+\varepsilon)\right|^{p}-\left|X \log (\Tilde{u}_{1}+\varepsilon)\right|^{p}\big) d x\\
& -\int_{\Omega} p (\Tilde{u}_{1}+\varepsilon)^{p}\left|X \log (u_{1}+\varepsilon)\right|^{p-2} X \log (u_{1}+\varepsilon) \cdot\big(X \log (\Tilde{u}_{1}+\varepsilon)-X \log (u_{1}+\varepsilon)\big) d x.
\end{aligned}   
\end{equation}
Combining \eqref{simplicity: convexity (a) vu} and \eqref{simplicity: convexity (a) uv}, then using \eqref{long calculation} we have
\begin{equation*}
\begin{aligned}
&0  \leq 2C \int_{\Omega}\left(\frac{1}{(\Tilde{u}_{1}+\varepsilon)^{p}}+\frac{1}{(u_{1}+\varepsilon)^{p}}\right)\big|(\Tilde{u}_{1}+\varepsilon) X u_{1}-(u_{1}+\varepsilon) X \Tilde{u}_{1}\big|^{p} d x \\
& \leq\lambda_{1} \int_{\Omega}\left(\left(\frac{u_{1}}{u_{1}+\varepsilon}\right)^{p-1}-\left(\frac{\Tilde{u}_{1}}{\Tilde{u}_{1}+\varepsilon}\right)^{p-1}\right)\big((u_{1}+\varepsilon)^{p}-(\Tilde{u}_{1}+\varepsilon)^{p}\big) d x.
\end{aligned}    
\end{equation*}
Observe that  
\begin{equation}\label{simplicity: dominated convergence}
\lim _{\varepsilon \rightarrow 0+} \int_{\Omega}\left(\left(\frac{u_{1}}{u_{1}+\varepsilon}\right)^{p-1}-\left(\frac{\Tilde{u}_{1}}{\Tilde{u}_{1}+\varepsilon}\right)^{p-1}\right)\big((u_{1}+\varepsilon)^{p}-(\Tilde{u}_{1}+\varepsilon)^{p}\big) d x=0    .
\end{equation}
Taking into account \eqref{simplicity: dominated convergence}, we apply the Fatou lemma to obtain
\begin{equation*}
\int_{\Omega}\left(\frac{1}{\Tilde{u}_{1}^{p}}+\frac{1}{u_{1}^{p}}\right)\big|\Tilde{u}_{1} X u_{1}-u_{1} X \Tilde{u}_{1}\big|^{p} d x=0.    
\end{equation*}
This is equivalent to  $X\left(u_{1}/\Tilde{u}_{1}\right)=0$ a.e. in $\Omega$.  Therefore,  $u_{1}=c\Tilde{u}_{1}$ a.e. in $\Omega$ provided $c\neq 0$. Then $u_{1}=c\Tilde{u}_{1}$ in $\Omega$ by Theorem \ref{Holder continuity}.

The proof of the case $1<p<2$ is similar. Multiplying the inequality \eqref{convexity (b)} with $\omega_{1}=X \log (\Tilde{u}_{1}+\varepsilon)$ and $\omega_{2}=X \log (u_{1}+\varepsilon)$ by $(u_{1}+\varepsilon)^p$ and then integrating over $\Omega$ we obtain
\begin{equation}\label{simplicity: convexity (b) vu}
\begin{aligned}
0  &\leq C \int_{\Omega}\frac{1}{(\Tilde{u}_{1}+\varepsilon)^{p}}\frac{\big|(\Tilde{u}_{1}+\varepsilon) X u_{1}-(u_{1}+\varepsilon) X \Tilde{u}_{1}\big|^{2}}{\big((\Tilde{u}_{1}+\varepsilon)\left|X u_{1}\right|+(u_{1}+\varepsilon)\left|X \Tilde{u}_{1}\right|\big)^{2-p}} d x  \\
& \leq \int_{\Omega}(u_{1}+\varepsilon)^{p}\big(\left|X \log (u_{1}+\varepsilon)\right|^{p}-\left|X \log (\Tilde{u}_{1}+\varepsilon)\right|^{p}\big) d x\\
&-\int_{\Omega} p (u_{1}+\varepsilon)^{p}\left|X \log (\Tilde{u}_{1}+\varepsilon)\right|^{p-2} X \log (\Tilde{u}_{1}+\varepsilon) \cdot\big(X \log (u_{1}+\varepsilon)-X \log (\Tilde{u}_{1}+\varepsilon)\big) d x.
\end{aligned}
\end{equation}
Multiplying the inequality \eqref{convexity (b)} with $\omega_{1}=X \log (u_{1}+\varepsilon)$ and $\omega_{2}=X \log (\Tilde{u}_{1}+\varepsilon)$ by $(\Tilde{u}_{1}+\varepsilon)^p$ and then integrating over $\Omega$ we obtain
\begin{equation}\label{simplicity: convexity (b) uv}
\begin{aligned}
0 & \leq C \int_{\Omega}\frac{1}{(u_{1}+\varepsilon)^{p}}\frac{\big|(\Tilde{u}_{1}+\varepsilon) X u_{1}-(u_{1}+\varepsilon) X \Tilde{u}_{1}\big|^{2}}{\big((\Tilde{u}_{1}+\varepsilon)\left|X u_{1}\right|+(u_{1}+\varepsilon)\left|X \Tilde{u}_{1}\right|\big)^{2-p}} d x \\
& \leq-\int_{\Omega}(\Tilde{u}_{1}+\varepsilon)^{p}\big(\left|X \log (u_{1}+\varepsilon)\right|^{p}-\left|X \log (\Tilde{u}_{1}+\varepsilon)\right|^{p}\big) d x\\
& -\int_{\Omega} p (\Tilde{u}_{1}+\varepsilon)^{p}\left|X \log (u_{1}+\varepsilon)\right|^{p-2} X \log (u_{1}+\varepsilon) \cdot\big(X \log (\Tilde{u}_{1}+\varepsilon)-X \log (u_{1}+\varepsilon)\big) d x.
\end{aligned}   
\end{equation}
Combining \eqref{simplicity: convexity (b) vu} and \eqref{simplicity: convexity (b) uv}, then using \eqref{long calculation} we have
\begin{equation}\label{simplicity 1<p<2}
\begin{aligned}
0 & \leq 2C \int_{\Omega}\left(\frac{1}{(\Tilde{u}_{1}+\varepsilon)^{p}}+\frac{1}{(u_{1}+\varepsilon)^{p}}\right) \frac{\big|(\Tilde{u}_{1}+\varepsilon) X u_{1}-(u_{1}+\varepsilon) X \Tilde{u}_{1}\big|^{2}}{\big((\Tilde{u}_{1}+\varepsilon)\left|X u_{1}\right|+(u_{1}+\varepsilon)\left|X \Tilde{u}_{1}\right|\big)^{2-p}} d x \\
& \leq\lambda_{1} \int_{\Omega}\left(\left(\frac{u_{1}}{u_{1}+\varepsilon}\right)^{p-1}-\left(\frac{\Tilde{u}_{1}}{\Tilde{u}_{1}+\varepsilon}\right)^{p-1}\right)\big((u_{1}+\varepsilon)^{p}-(\Tilde{u}_{1}+\varepsilon)^{p}\big) d x.
\end{aligned}    
\end{equation}
Again, the right-hand side of \eqref{simplicity 1<p<2} tends to zero as $\varepsilon\to 0$. So, the Fatou lemma ensures $X\left(u_{1}/\Tilde{u}_{1}\right)=0$ a.e. in $\Omega$. Then $u_{1}=c\Tilde{u}_{1}$ a.e. in $\Omega$ provided $c\neq 0$. Then $u_{1}=c\Tilde{u}_{1}$ in $\Omega$ by Theorem \ref{Holder continuity}.
\end{proof}

The techniques used in Theorem \ref{simplicity theorem} allow us to prove Theorem \ref{change sign theorem} and Theorem \ref{isolated theorem} (cf \cite{AL1987}).
\begin{theorem}\label{change sign theorem}
Let $\lambda\neq \lambda_{1}$ be an eigenvalue. Then the corresponding eigenfunction changes sign in $\Omega$.
\end{theorem}
\begin{proof}
Let $u$ be an eigenfunction corresponding to an eigenvalue $\lambda\neq\lambda_{1}$, i.e., $\lambda> \lambda_{1}$. Suppose by contradiction $u\geq 0$ in $\Omega$. Proceeding the same procedure as in the proof of Theorem \ref{simplicity theorem} we obtain
\begin{equation*}
\int_{\Omega}\left(\lambda_{1}\left(\frac{u_{1}}{u_{1}+\varepsilon}\right)^{p-1}-\lambda\left(\frac{u}{u+\varepsilon}\right)^{p-1}\right)\big((u_{1}+\varepsilon)^{p}-(u+\varepsilon)^{p}\big) d x \geq 0.    
\end{equation*}
Then
\begin{equation*}
\left(\lambda_{1}-\lambda\right)\int_{\Omega}\big(u_{1}^{p}-u^{p}\big) d x \geq 0  
\end{equation*}
as $\varepsilon\to 0$. Thanks to Theorem \ref{simplicity theorem}, one can multiply $u_{1}$ by a sufficiently large positive constant $C$ to derive $u_{1}^{p}-u^{p}>0$ in $\Omega$. However, this leads to a contradiction. Thus, $u$ changes sign in $\Omega$.
\end{proof}

\begin{theorem}\label{isolated theorem}
The first eigenvalue $\lambda_{1}$ is isolated.
\end{theorem}
\begin{proof}
Suppose, by contradiction,  $\lambda_{1}$ is a limit point of the spectrum, i.e., there exists  a sequence of eigenvalues $\{\Tilde{\lambda}_{j}\}$ which converges to $\lambda_{1}$
 and let $\{\Tilde{u}_{j}\}$ be  corresponding eigenfunctions with $\|\Tilde{u}_{j}\|_{p}=1$. Hence,
\begin{equation*}
\Tilde{\lambda}_{j}=\int_{\Omega}|X\Tilde{u}_{j}|^{p}dx. 
\end{equation*}
Boundedness of $\{\Tilde{\lambda}_{j}\}$ implies the boundedness of $\{\Tilde{u}_{j}\}$ in $\mathcal{W}_{X,0}^{1, p}(\Omega)$. Therefore, by the Eberlein-\v{S}mulian theorem, Proposition \ref{Weak convergence characterization} and Theorem \ref{compact Sobolev embedding}, there exist a subsequence $\{\Tilde{u}_{j_{k}}\}$ and a function $\Tilde{u}\in \mathcal{W}_{X,0}^{1, p}(\Omega)$ such that
\begin{equation*}
\Tilde{u}_{j_{k}} \to \Tilde{u}\enspace \text{ strongly in } L^{p}(\Omega)\quad \text{and} \quad X \Tilde{u}_{j_{k}} \rightharpoonup  X \Tilde{u} \enspace \text{ weakly in } L^{p}(\Omega). 
\end{equation*}
In view of the weakly lower semicontinuity of $\|\cdot\|_{p}$ norm, we have
\begin{equation*}
\int_{\Omega}|X \Tilde{u}|^{p}dx \leq \lim_{j_{k}\to\infty}\Tilde{\lambda}_{j_{k}} =\lambda_{1}, 
\end{equation*}
so $\Tilde{u}$ minimizes the Rayleigh quotient and $\Tilde{u}$ is the first eigenfunction by Theorem \ref{existence and uniqueness of the Rayleigh quotient}. Then $\Tilde{u}>0$ in $\Omega$ by Corollary \ref{positivity of the first eigenfuncion} .

If  $\Tilde{\lambda}_{j_{k}} \neq \lambda_{1}$, then $\Tilde{u}_{j_{k}}$ changes signs in $\Omega$ by Theorem \ref{change sign theorem}. Therefore, $$\Omega_{j_{k}}^{+}:=\left\{x\in\Omega: \enspace \Tilde{u}_{j_{k}}(x)>0\right\}\quad \text{and} \quad \Omega_{j_{k}}^{-}:=\left\{x\in\Omega: \enspace \Tilde{u}_{j_{k}}(x)<0\right\}$$ 
are nonempty sets. Let us show that
\begin{equation}\label{lower bound of Tilde{u}_{j_{k}}}
    \left|\Omega_{j_{k}}^{+}\right| \geq \left(C\Tilde{\lambda}_{j_{k}}\right)^{\frac{p^{*}}{p^{*}-p}},
\end{equation}
where $p^{*}$ is expressed as \eqref{p asterisk}. Indeed, we can view  $\Tilde{u}_{j_{k}}^{+}=\max\{\Tilde{u}_{j_{k}}, 0\}$
as a test function by Proposition \ref{proposition Xu}, thus
\begin{equation*}
\begin{aligned}
\int_{\Omega_{j_{k}}^{+}}|X \Tilde{u}_{j_{k}}|^{p} d x &=\int_{\Omega}|X \Tilde{u}_{j_{k}}^{+}|^{p} d x=\int_{\Omega}|X \Tilde{u}_{j_{k}}|^{p-2} X \Tilde{u}_{j_{k}} \cdot X \Tilde{u}_{j_{k}}^{+} d x\\
&=\Tilde{\lambda}_{j_{k}} \int_{\Omega}|\Tilde{u}_{j_{k}}|^{p-2} \Tilde{u}_{j_{k}} \Tilde{u}_{j_{k}}^{+} d x=\Tilde{\lambda}_{j_{k}}\int_{\Omega}|\Tilde{u}_{j_{k}}^{+}|^{p} d x=\Tilde{\lambda}_{j_{k}}\int_{\Omega_{j_{k}}^{+}}|\Tilde{u}_{j_{k}}|^{p} d x.
\end{aligned}
\end{equation*}
Using the H\"older inequality and  Sobolev embedding \eqref{embeddings for lower bound of eigenvalues}, we obtain
\begin{equation*}
\begin{aligned}
\int_{\Omega_{j_{k}}^{+}}|X \Tilde{u}_{j_{k}}|^{p} d x= \Tilde{\lambda}_{j_{k}}\int_{\Omega_{j_{k}}^{+}}|\Tilde{u}_{j_{k}}|^{p} d x\leq \Tilde{\lambda}_{j_{k}}|\Omega_{j_{k}}^{+}|^{1-\frac{p}{p^{*}}}\left(\int_{\Omega}|\Tilde{u}_{j_{k}}^{+}|^{p^{*}}dx\right)^{\frac{p}{p^{*}}}\\
\leq C\Tilde{\lambda}_{j_{k}}|\Omega_{j_{k}}^{+}|^{1-\frac{p}{p^{*}}}\int_{\Omega_{j_{k}}^{+}}|X \Tilde{u}_{j_{k}}|^{p} d x,
\end{aligned}
\end{equation*}
which proves \eqref{lower bound of Tilde{u}_{j_{k}}}. Similarly, one can prove that
\begin{equation*}
\left|\Omega_{j_{k}}^{-}\right| \geq \left(C\Tilde{\lambda}_{j_{k}}\right)^{\frac{p^{*}}{p^{*}-p}} 
\end{equation*}
with the test function $\Tilde{u}_{j_{k}}^{-}=\min\{\Tilde{u}_{j_{k}}, 0\}$.
Consequently, we obtain
$$\left|\Omega^{\pm}\right|=\left|\limsup\Omega_{j_{k}}^{\pm}\right|> 0.$$ 
By \cite[Theorem 4.9]{Bre2011},  there exists $\{\Tilde{u}_{j_{k_{l}}}\}\subset \{\Tilde{u}_{j_{k}}\}$ such that $\Tilde{u}_{j_{k_{l}}}\to \Tilde{u}$ a.e. in $\Omega$. It turns out that $\Tilde{u}>0$ a.e. in $\Omega^{+}$ and $\Tilde{u}<0$ a.e. in $\Omega^{-}$, i.e., the first eigenfunction $\Tilde{u}$ changes sign in $\Omega$. This is a contradiction, so $\lambda_{1}$ is isolated.
\end{proof}

\section{Appendix}
The positive and negative parts of a function $u$ are given by
$$u^{+}:=\max \{u, 0\}\quad\text{and}\quad u^{-}:=\min \{u, 0\}.$$
\begin{proposition}\label{proposition Xu} \cite[Lemma 3.5]{garofalo1996isoperimetric}. 
Let $u \in \mathcal{W}_{X,0}^{1, p}(\Omega)$. Then  $u^{\pm},|u| \in \mathcal{W}_{X,0}^{1, p}(\Omega)$ and 
\begin{equation*}
X|u|=\begin{cases}
Xu, & \text{if $u>0$ a.e.  in $\Omega$,}\\
0, & \text{if $u= 0$ a.e.  in $\Omega$,}\\
-Xu & \text{if $u<0$ a.e.  in $\Omega$.}
\end{cases}
\end{equation*}
\end{proposition}

\begin{proposition}\label{convexity inequality}\cite[Lemma 4.2]{Lin1990}
\begin{enumerate}[label=(\roman*)]
\item  Let $1<p<2$. Then there exists a constant $C>0$ such that
\begin{equation}\label{convexity (b)}
|\omega_{2}|^{p} \geq|\omega_{1}|^{p}+p|\omega_{1}|^{p-2} \omega_{1} \cdot(\omega_{2}-\omega_{1})+ \frac{C|\omega_{2}-\omega_{1}|^{2}}{(|\omega_{2}|+|\omega_{1}|)^{2-p}} 
\end{equation}
holds  for all $\omega_{1}, \omega_{2} \in \mathbb{R}^{m}$. The constant $C$ depends on only $p$.

\item  Let $p \geq 2$. Then there exists a constant $C>0$ such that
\begin{equation}\label{convexity (a)}
|\omega_{2}|^{p} \geq|\omega_{1}|^{p}+p|\omega_{1}|^{p-2} \omega_{1} \cdot(\omega_{2}-\omega_{1})+C|\omega_{2}-\omega_{1}|^{p}
\end{equation}
holds for all $\omega_{1}, \omega_{2} \in \mathbb{R}^{m}$. The constant $C$ depends only on $p$.
\end{enumerate}
\end{proposition}

\section{Acknowledgments}

This research was funded by the Science Committee of the Ministry of Science and Higher Education of Kazakhstan (Grant No. AP19674900). This work was also supported by Nazarbayev University grant 20122022FD4105. No new data were collected or generated during the course of research. We thank an anonymous referee whose comments and suggestions helped improve the clarity of our exposition.

\addcontentsline{toc}{chapter}{Bibliography}
\bibliography{refs}      

\bibliographystyle{abbrv}  

\end{document}